\documentclass[11pt,reqno]{amsart}
\usepackage{enumerate}
\usepackage{url}
\usepackage{cite}
\usepackage{comment}
\usepackage{graphicx}
\usepackage{upgreek}
\usepackage{tikz}
\usetikzlibrary{arrows,shapes,trees,backgrounds}
\usepackage{fancyhdr}
\usepackage{amsfonts, amsmath, amssymb, amscd, amsthm, bm, cancel}
\usepackage[
linktocpage=true,colorlinks,citecolor=magenta,linkcolor=blue,urlcolor=magenta]{hyperref}
\usepackage[margin=1.1in]{geometry}

\newtheorem{prop}{Proposition}[section]
\newtheorem{thm}[prop]{Theorem}
\newtheorem{lem}[prop]{Lemma}
\newtheorem{cor}[prop]{Corollary}

\theoremstyle{definition}
\newtheorem*{ack}{Acknowledgments}
\theoremstyle{remark}
\newtheorem{rem}[prop]{Remark}

\numberwithin{equation}{section}
\allowdisplaybreaks

\begin{document}

\title[Anisotropic isoperimetric inequalities]{Anisotropic weighted isoperimetric inequalities for star-shaped and $F$-mean convex hypersurface}

\author[R. Zhou]{Rong Zhou}
\address{School of Mathematical Sciences, University of Science and Technology of China, Hefei 230026, P.R. China}
\email{\href{mailto:zhourong@mail.ustc.edu.cn}{zhourong@mail.ustc.edu.cn}}
\author[T. Zhou]{Tailong Zhou}
\address{School of Mathematics, Sichuan University, Chengdu 610065, Sichuan, P.R. China}
\email{\href{mailto:zhoutailong@scu.edu.cn}{zhoutailong@scu.edu.cn}}
\date{\today}

\begin{abstract}
We prove two anisotropic type weighted geometric inequalities that hold for star-shaped and $F$-mean convex hypersurfaces in $\mathbb{R}^{n+1}$. These inequalities involve the anisotropic $p$-momentum, the anisotropic perimeter and the volume of the region enclosed by the hypersurface. We show that the Wulff shape of $F$ is the unique minimizer of the corresponding functionals among all star-shaped and $F$-mean convex sets. We also consider their quantitative versions characterized by asymmetry index and the Hausdorff distance between the hypersurface and a rescaled Wulff shape. As a corollary, we obtain the stability of the Weinstock inequality for the first non-zero Steklov eigenvalue for star-shaped and strictly mean convex domains.
\end{abstract}

\maketitle



\section{Introduction}

\subsection{Anisotropic weighted isoperimetric inequalities}

The anisotropic version of the classic Euclidean isoperimetric problem is the Wulff Problem, which is also known as the Equilibrium Shape Problem: Consider
\begin{equation*}
	\inf\left\{\int_{\partial E}F(\nu)d\mathcal{H}^{n}:\text{Vol}(E)=m\right\},\ m>0,
\end{equation*}
where $F:\mathbb{R}^{n+1}\rightarrow[0,\infty)$ is 1-homogeneous, convex with $F(x)\geq c|x|$ uniformly for some $c>0$, $E\subset\mathbb{R}^{n+1}$ is a bounded open set with $C^1$-boundary, and $\nu$ is the unit outward normal of $\partial E$. 
The German crystallographer G. Wulff first guessed in \cite{Wul01} that the unique minimizer of the Wulff Problem, up to translations and scalings, is the domain $L$ bounded by $\mathcal{W}$ the Wulff shape of $F$:
\begin{equation*}
	\partial L=\mathcal{W}=\partial\left(\cap_{y\in\mathbb{S}^n}\left\{x\in\mathbb{R}^{n+1}:x\cdot y<F(y)\right\}\right).
\end{equation*}
When considering the Wulff Problem for all measurable sets $E$ of finite perimeter, the proof of the minimality of the Wulff shape can be found e.g. in \cite{Maggi}, while the uniqueness was first obtained by Taylor in \cite{Taylor,Taylor'}. Relevant uniqueness results, especially when $F$ is a general integrand, can be found in \cite{JBro-FMor,IF-SM,HeLiMaGe,ADRosa-K-S}. 

Besides the classic isoperimetric inequality, we are interested to study the anisotropic version of weighted geometric inequalities for hypersurfaces in $\mathbb{R}^{n+1}$. As a corollary of the divergence Theorem, it is known that
\begin{equation}\label{ineq-weight}
	\int_{\Sigma}rd\mu\geqslant (n+1)\text{Vol}(\Omega)
\end{equation}
for all closed orientable embedded hypersurface $\Sigma=\partial \Omega$, where $r$ is the distance to the origin. Equality of \eqref{ineq-weight} holds if and only if $\Sigma$ is a sphere centered at the origin. Gir\~ao-Rodrigues \cite{FGirao-DRodrigues} pointed out that \eqref{ineq-weight} can't be improved to $\int_{\Sigma}rd\mu\geqslant|\mathbb{S}^n|\left(\frac{|\Sigma|}{|\mathbb{S}^n|}\right)^{\frac{n+1}{n}}$ even for strictly convex $\Sigma$. Instead, by using the inverse mean curvature flow to construct monotonous quantities, they can improve the inequality to
\begin{equation}\label{ineq-weight2}
	\int_{\Sigma}rd\mu\geqslant \frac{n}{n+1}|\mathbb{S}^n|\left(\frac{|\Sigma|}{|\mathbb{S}^n|}\right)^{\frac{n+1}{n}}+\text{Vol}(\Omega)
\end{equation}
for star-shaped and strictly mean convex hypersurface $\Sigma$. 

The inverse anisotropic mean curvature flow for star-shaped hypersurface with positive anisotropic mean curvature was studied by Xia in \cite{CXia17} for $n\geqslant 2$. He proved that the flow has a smooth solution and converges exponentially in $C^{\infty}$ to the Wulff shape after rescaling. When $n=1$, the flow becomes an evolution of embedded closed convex curves in the plane $\mathbb{R}^2$, which has been studied by Andrews in \cite{BAndrews98}. With the help of the inverse anisotropic mean curvature flow, we are able to generalize \eqref{ineq-weight2} to the anisotropic case:
\begin{thm}\label{thm-iso-ineq1}
	Let $\Sigma=\partial\Omega\subset\mathbb{R}^{n+1}$, $n\geq1$, be a smooth closed hypersurface which is star-shaped w.r.t. the origin and $F$-mean convex (i.e., the anisotropic mean curvature $H_F\geqslant 0$), then for any point $P\in\mathbb{R}^{n+1}$,
	\begin{equation}\label{eq-iso-ineq1}
		\int_{\Sigma}F^0(x-P)d\mu_F\geqslant\frac{n}{(n+1)^{1+\frac{1}{n}}}\mathrm{Vol}(L)^{-\frac{1}{n}}\big|\Sigma\big|_F^{1+\frac{1}{n}}+\mathrm{Vol}(\Omega),
	\end{equation}
where $L$ is the enclosed domain by the Wulff shape $\mathcal{W}$, $\big|\Sigma\big|_F$ is the anisotropic perimeter of $\Sigma$. And equality is attained if and only if $\Sigma$ is a rescaling of $\mathcal{W}$ centered at $P$. 
\end{thm}
Our Theorem \ref{thm-iso-ineq1} also shows that the strict mean convexity condition for the weighted inequalities \eqref{ineq-weight2} can be weakened, i.e., the mean curvature of $\Sigma$ can be non-negative.

Recently, Kwong-Wei \cite{Kwong-Wei} considered sharp inequalities for weighted integrals of the form $\int_{\Sigma}r^kd\mu$, which correspond to the Euclidean weighted isoperimetric problems. Similar to \eqref{ineq-weight2}, their inequalities involve the perimeter of $\Sigma$ and (weighted) volume of $\Omega$, and hold for space forms and some general warped product spaces. 

In case $k=2$, the weighted integral $\int_{\Sigma}r^2d\mu$ is known as the polar moment of inertia, an important quantity in Newtonian physics. Moreover, the following form of isoperimetric inequality for this weighted integral
\begin{equation}\label{ineq-crucial}
	\int_{\Sigma}r^2d\mu\geqslant\text{Vol}(\mathbb{B}^{n+1})^{-\frac{2}{n+1}}|\Sigma|\text{Vol}(\Omega)^{\frac{2}{n+1}}
\end{equation}
is crucial for the proof of the Weinstock inequality by Bucur-Ferone-Nitsch-Trombetti in \cite{B-F-N-T} and by Kwong-Wei in \cite{Kwong-Wei} with weaker convexity assumption. Their results confirmed that for a bounded, open and star-shaped domain $\Omega$ with mean convex boundary $\Sigma$, the Weinstock inequality holds:
\begin{equation}
	\sigma_1(\Omega)|\Sigma|^{\frac{1}{n}}\leqslant\sigma_1(\mathbb{B}^{n+1})|\mathbb{S}^n|^{\frac{1}{n}},  \label{weinstock}
\end{equation} 
where $\sigma_1(\cdot)$ is the first nontrivial Steklov eigenvalue of a region:
\begin{equation}
	\sigma_1(\Omega) = \min\left\lbrace \dfrac{\int_{\Omega} \left| Du \right|^2 \mathrm{d}x}{\int_{\partial\Omega} u^2 \mathrm{d}\mu}:u\in H^1(\Omega)\backslash\{0\},\int_{\partial\Omega} u \mathrm{d}\mu = 0 \right\rbrace.
	\label{steklov}
\end{equation}

Paoli-Trani \cite{G.Paoli-LTrani} studied the anisotropic generalization of the inequality \eqref{ineq-crucial} involving the $p$-momentum, looking forward for possible applications to the Steklov spectrum problem for the pseudo $p$-Laplacian (see \cite{B-F}): For a bounded, open convex set $\Omega\subset\mathbb{R}^{n+1}$, with boundary $\Sigma$, $p>1$, there holds
	\begin{equation}\label{ineq-aniso2}
	\int_{\Sigma}\left(F^0(x)\right)^pd\mu_F\geqslant\text{Vol}(L)^{-\frac{p}{n+1}}\big|\Sigma\big|_F\text{Vol}(\Omega)^{\frac{p}{n+1}},		
\end{equation}
where equality holds only for the rescalings of the Wulff shape $\mathcal{W}$. Their proof relies on the inverse anisotropic mean curvature flow and the anisotropic Heintze-Karcher inequality.  

As a corollary of our Theorem \ref{thm-iso-ineq1}, we prove that the conditions in Paoli-Trani's inequality \eqref{ineq-aniso2} for the $p$-momentum can be weakened to star-shaped, $F$-mean convex hypersurface with $p\geqslant 1$:
\begin{thm}\label{thm-iso-ineq2}
	Under the same conditions in Theorem \ref{thm-iso-ineq1}, let $p\geq1$, then
	\begin{equation}\label{ineq-aniso2'}
		\int_{\Sigma}\left(F^0(x-P)\right)^pd\mu_F\geqslant\mathrm{Vol}(L)^{-\frac{p}{n+1}}\big|\Sigma\big|_F\mathrm{Vol}(\Omega)^{\frac{p}{n+1}}	
	\end{equation} 
 holds and equality is attained if and only if $\Sigma$ is a rescaling of the Wulff shape $\mathcal{W}$ centered at $P$.
\end{thm}

\subsection{Stability of Weinstock inequality}

We next consider the stability of anisotropic inequalities \eqref{eq-iso-ineq1} and \eqref{ineq-aniso2}, which are characterized by the asymmetry index and the Hausdorff distance between the hypersurface and a rescaled Wulff shape.

Recall that for any two bounded domains $E,\ G\subset \mathbb{R}^{n+1}$, the symmetric difference and Hausdorff distance between them are defined as
\begin{equation}
	|E\triangle G|=\mathrm{Vol}(E\backslash G)+\mathrm{Vol}(G\backslash E),
\end{equation}
\begin{equation}\label{def-aHdist}
	\mathrm{dist}(\partial E,\partial G) = \inf \left\lbrace r\geq0:G\subset E+r\mathbb{B}\ \mbox{and}\ E \subset G+r\mathbb{B}\right\rbrace,
\end{equation}
where $\mathbb{B}$ is the unit ball and ``$+$" is the Minkowski sum. The asymmetry index $\alpha_F(E)$ (also known as the Fraenkel asymmetry $\alpha(E)$ in isotropic case) of the set $E$ is defined as the minimum of symmetric difference between $E$ and any domain bounded by the rescaled and translated Wulff shape with the same volume
\begin{equation}\label{def-asymmetry-index}
	\alpha_F(E):=\min_{P\in\mathbb{R}^{n+1}}\left\{\frac{|E\triangle L_a(P)|}{\mathrm{Vol}(E)}:\mathrm{Vol}(E)=\mathrm{Vol}(L_a(P))\right\},
\end{equation}
 where we denote the rescaled Wulff shape centered at a point $P$ and the corresponding domain as $\mathcal{W}_a(P)=a\mathcal{W}+P$ and $L_a(P)=aL+P$ for simplicity.

Figalli, Maggi and Pratelli \cite{figalli} used optimal transport to show that there exists a constant $C(n)$ such that
	\begin{equation}
		\alpha_F(E)^2 \leqslant C(n) \dfrac{|\partial E|_{F}}{|\partial L|_{F} \left( \frac{|E|}{|L|}\right)^{\frac{n}{n+1}}}-1 \label{equ-quantiwulff}
	\end{equation}
for any set of finite perimeter $E$ with $0<|E|<\infty$. Later, Neumayer and Robin \cite{NR} considered a strong form of the quantitative Wulff inequality and they proved the analogue of Fuglede's stablility results for isoperimetric inequality of nearly spherical sets in the anisotropic case (see \cite[Proposition 1.9]{NR}). Here we consider the stability of (\ref{eq-iso-ineq1}) and (\ref{ineq-aniso2'}) for star-shaped and strictly $F$-mean convex hypersurfaces charaterized by the $L^1$ distance.
\begin{thm}\label{thm-L1-stability}
Let $\Sigma$ be a smooth closed hypersurface enclosing a domain $\Omega$, which is star-shaped w.r.t. the origin and strictly $F$-mean convex. Suppose that $B_{\rho_{-}}(P)\subset\Omega\subset B_{\rho_{+}}(P)$. Then there exists a positive constant $C=C(n,F,\rho_{-},\rho_{+},\min\limits_{\Sigma} H_F )$ such that
	\begin{equation}
		\alpha_F(\Omega) \leqslant C\cdot f_1\left[ \big|\Sigma\big|_F^{-(1+\frac{1}{n})}\left( \int_{\Sigma} F^0(x-P)\mathrm{d}\mu_F - \mathrm{Vol}(\Omega) \right)- (n+1)^{-(1+\frac{1}{n})}n\mathrm{Vol}(L)^{-\frac{1}{n}} \right], \label{L1-stable1}
	\end{equation}
	and
	\begin{equation}
		\alpha_F(\Omega) \leqslant C\cdot f_1\left[ \dfrac{\displaystyle\int_{\Sigma}\left(F^0(x-P)\right)^p\mathrm{d}\mu_F}{\big|\Sigma\big|_F\mathrm{Vol}(\Omega)^{\frac{p}{n+1}}}-\mathrm{Vol}(L)^{-\frac{p}{n+1}} \right],   \label{L1-fostable}
	\end{equation}
	where $f_1(s)=s^{\frac{1}{4}}+\sqrt{s}\ (s\geqslant 0)$.
\end{thm}

	In \cite{Scheuer}, Scheuer used curvature flows and properties of nearly umbilical hypersurfaces to obtain the stability of quermassintegral inequalities characterized by the Hausdorff distance between the non-convex, star-shaped hypersurface and a geodesic sphere in Euclidean space $\mathbb{R}^{n+1}$. Then in \cite{Sahjwani-Scheuer}, Sahjwani and Scheuer obtained such stability result for horo-convex hypersurfaces in hyperbolic space $\mathbb{H}^{n+1}$. Stability of inequalities involving anisotropic curvature functionals were also obtained in \cite{Sche} by Scheuer and Zhang. When the hypersurface $\Sigma$ is strictly $F$-mean convex and star-shaped with a $C^1$ bound, we obtain the stability of inequalities \eqref{eq-iso-ineq1} and \eqref{ineq-aniso2'} for the Hausdorff distance between $\Sigma$ and a proper rescaled Wulff shape.
\begin{thm}\label{thm-aniso-stability}
	Let $\Sigma$ be a smooth closed hypersurface enclosing a domain $\Omega$, which is star-shaped w.r.t. the origin and strictly $F$-mean convex. Regard $\Sigma$ as a graph on $\mathbb{S}^n$ under the polar coordinate system, that is $\Sigma=\{(r(\theta),\theta):\theta\in\mathbb{S}^n\}$. Then there exists a positive constant $C=C(n,F,\min\limits_{\mathbb{S}^n}r,\max\limits_{\mathbb{S}^n}\left|\nabla_{\mathbb{S}^n}r\right|,\min\limits_{\Sigma} H_F )$ and $a>0$ such that
	\begin{equation}
		\mathrm{dist}(\Sigma,a\mathcal{W}) \leqslant C\cdot f_2\left[\big|\Sigma\big|_F^{-(1+\frac{1}{n})}\left( \int_{\Sigma} F^0(x)\mathrm{d}\mu_F - \mathrm{Vol}(\Omega) \right)- (n+1)^{-(1+\frac{1}{n})}n\mathrm{Vol}(L)^{-\frac{1}{n}} \right], \label{stable1}
	\end{equation}
	and
	\begin{equation}
		\mathrm{dist}(\Sigma,a\mathcal{W}) \leqslant C\cdot f_2\left[ \dfrac{\displaystyle\int_{\Sigma}\left(F^0(x)\right)^p\mathrm{d}\mu_F}{\big|\Sigma\big|_F\mathrm{Vol}(\Omega)^{\frac{p}{n+1}}}-\mathrm{Vol}(L)^{-\frac{p}{n+1}} \right],   \label{fostable}
	\end{equation}
	where $f_2(s):=s^{\frac{1}{2(n+2)}}+\sqrt{s}\ (s\geq0)$.
\end{thm}


In \cite{Gavitone}, Gavitone, La Manna, Paoli and Trani proved a quantitative Weinstock inequality for convex sets, they attributed the proof to the case of nearly spherical sets by the selection principle. As a corollary of  Theorem \ref{thm-L1-stability} and \ref{thm-aniso-stability}, we proof the stability of the Weinstock inequality for the first non-zero Steklov eigenvalue for star-shaped and strictly mean convex domains, which is the quantitative version of (\ref{weinstock}) proved in \cite{Kwong-Wei}. 
\begin{cor}\label{cor-weinstock-stability}
	Let $\Omega\subset\mathbb{R}^{n+1}$ be a bounded domain with smooth, star-shaped and strictly mean convex boundary. Denote $\mathbb{B}^{n+1}$ as the unit ball in $\mathbb{R}^{n+1}$. Let $P$ be the barycenter of $\Omega$.\\
	(i)Suppose that $B_{\rho_{-}}(P)\subset\Omega\subset B_{\rho_{+}}(P)$. Define $f_1(s)=s^{\frac{1}{4}}+\sqrt{s}$ on $\mathbb{R}^+$. Then there exsits a positive constant $C=C(n,\rho_{-},\rho_{+},\min\limits_{\Sigma} H) $ 
	such that
	\begin{equation}
			 \alpha(\Omega) \leqslant C\cdot f_1\left[ \sigma_1(\mathbb{B}^{n+1}) \left|\mathbb{S}^n\right|^{\frac{1}{n}} - \sigma_1(\Omega) \left|\partial \Omega\right|^{\frac{1}{n}} \right].  \label{L1-weinstable}
	\end{equation}\\
(ii)Suppose that $\Sigma$ is star-shaped w.r.t. $P$ and write $\partial\Omega=\{(r(\theta),\theta):\theta\in\mathbb{S}^n\}$ as a graph on $\mathbb{S}^n$. Define $f_2(s):=s^{\frac{1}{2(n+2)}}+\sqrt{s}$ on $\mathbb{R}^+$. Then there exists a positive constant $C=C(n,\min\limits_{\mathbb{S}^n}r,\max\limits_{\mathbb{S}^n}\left|\nabla_{\mathbb{S}^n}r\right|,\min\limits_{\Sigma} H )$ and $a>0$ such that
	\begin{equation}
		\mathrm{dist}(\partial\Omega,S_a(P)) \leqslant C\cdot f_2\left[ \sigma_1(\mathbb{B}^{n+1}) \left|\mathbb{S}^n\right|^{\frac{1}{n}} - \sigma_1(\Omega) \left|\partial \Omega\right|^{\frac{1}{n}} \right]. \label{weinstable}
	\end{equation}
\end{cor}
\begin{rem}
 Except the case when all principal curvatures are 0, the strictly mean convex assumption is weaker than the convexity condition in \cite{Gavitone}. However, we could not generalize Theorem \ref{thm-L1-stability}, \ref{thm-aniso-stability} and Corollary \ref{cor-weinstock-stability} to ($F$-) mean convex case using our method. For instance, in isotropic case, if the mean curvature $H$ is only non-negative, along the level set solution of the weak inverse mean curvature flow, the flow hypersurface could jump discontinuously across a set of positive measure, which make it difficult to control the $L^1$, Hausdorff distance between the flow hypersurface and the initial hypersurface.    
\end{rem}

The structure of this paper is the following. In Section \ref{prelim}, we present some notations and basic properties of anisotropic geometry and hypersurface with anisotropic curvature, and introduce the existence and convergence results of the inverse anisotropic mean curvature flow. In Section \ref{sec-proof}, we prove the anisotropic weighted isoperimetric inequalities. In Section \ref{sec-proof-stability}, we prove the stability results.

\section{Preliminary}\label{prelim}
\subsection{Minkowski norm and Wulff shape}\label{sec-wulff}$\ $

We call $F:\mathbb{R}^{n+1}\rightarrow[0,\infty)$ a Minkowski norm if it is a norm in $\mathbb{R}^{n+1}$ (i.e., convex, 1-homogeneous and positive in $\mathbb{R}^{n+1}\backslash\{0\}$) and is smooth in $\mathbb{R}^{n+1}\backslash\{0\}$ with $D^2\left(\frac{1}{2}F^2\right)$ strictly positive definite.

The dual norm of $F$ is defined by
\begin{equation*}
F^0(x)=\sup_{y\neq0}\frac{x\cdot y}{F(y)}=\sup_{y\in\mathcal{K}}x\cdot y,\ \mathcal{K}=\{y\in\mathbb{R}^{n+1}:F(y)\leqslant 1\}.
\end{equation*}

For all $x,\ y\in\mathbb{R}^{n+1}\backslash\{0\}$, we have the following identities (see e.g. \cite{CXia13}):
	\begin{equation}\label{eq-F(DF^0)}
		F(DF^0(y))=1,\ F^0(DF(x))=1,
	\end{equation}
\begin{equation}
	DF\big|_{DF^0(y)}=\frac{y}{F^0(y)},\ DF^0\big|_{DF(x)}=\frac{x}{F(x)}. \label{eq-DF(DF^0)}
\end{equation}

The definition of $F^0$ implies the following anisotropic Cauchy-Schwartz inequality: 
\begin{equation}\label{ineq-CW}
	x\cdot y\leqslant F(x)F^0(y),\ \forall\ x,\ y\in\mathbb{R}^{n+1}.
\end{equation}
Moreover, equality of \eqref{ineq-CW} holds if and only if $\max_{z\neq 0}\frac{z\cdot y}{F(z)}$ is attained at $z=x$, which means $D_z\left(\frac{z\cdot y}{F(z)}\right)\big|_{z=x}=0$. Hence one can check that equality of \eqref{ineq-CW} holds if and only if $x=F^0(x)\cdot DF(y),\ y=F(y)\cdot DF^0(x)$.

Consider a smooth, compact and strictly convex hypersurface $\mathcal{W}$ in $\mathbb{R}^{n+1}$ which bounds a domain $L$, with the origin lying inside it. Let $F:\mathbb{S}^n\rightarrow\mathbb{R}^+$, $F(\xi)=\sup_{x\in\mathcal{W}}x\cdot \xi$ be the support function of $\mathcal{W}$. We can extend $F\in C^\infty(\mathbb{S}^n)$ homogeneously by
\begin{equation*}
	F(x)=|x|F\left(\frac{x}{|x|}\right),\ x\neq 0\ \text{and }F(0)=0.
\end{equation*}
One can check that $F$ becomes a Minkowski norm in $\mathbb{R}^{n+1}$. 

$\mathcal{W}$ is called the Wulff shape (or indicatrix) of $F$. The property of the support function and \eqref{eq-F(DF^0)} imply that
\begin{equation*}
	\mathcal{W}=\{ F(\xi)\xi+\nabla_{\mathbb{S}^n}F(\xi):\xi\in\mathbb{S}^n\}=\{DF(\xi):\xi\in\mathbb{S}^n\}=\{x\in\mathbb{R}^{n+1}:F^0(x)=1\}.
\end{equation*} 
\subsection{Anisotropic curvature of a hypersurface}\label{sec-a-curvature}$\ $

Let $(\Sigma,g)$ be a smooth hypersurface in $\mathbb{R}^{n+1}$, $\nu:\Sigma\rightarrow\mathbb{S}^n$ be its Gauss map. The anisotropic Gauss map $\nu_F:\Sigma\rightarrow\mathcal{W}$ is defined as
\begin{equation}
	\nu_F(x)=DF(\nu(x))=F(\nu(x))\nu(x)+\nabla_{\mathbb{S}^n}F(\nu(x)),
\end{equation}
then $\mathcal{W}$ and $\Sigma$ share the same unit (outer) normal $\nu(x)$ at $\nu_F(x)$ and $x$ respectively.

The anisotropic principal curvatures $\{\kappa_i^F\}_{i=1\cdots n}$ (w.r.t. $\mathcal{W}$) at $x\in\Sigma$ is defined as the eigenvalues of 
\begin{equation*}
	d\nu_F:T_x\Sigma\rightarrow T_{\nu_F(x)}\mathcal{W}.
\end{equation*}

We recall the following formulations of anisotropic curvatures in \cite{BAndrews01} by Andrews, which were reformulated by Xia in \cite{CXia13}, see also \cite{CXia17,CXia17'}. We introduce a Riemannian metric $G$ w.r.t. $F^0$ in $T\mathbb{R}^{n+1}$:
\begin{equation*}
	G(\xi)(V,W):=\frac{D^2\left(\frac{1}{2}\left(F^0(\xi)^2\right)\right)}{\partial\xi^\beta\partial\xi^\gamma}V^\beta W^\gamma,\ \xi\in\mathbb{R}^{n+1}\backslash\{0\},\ V,\ W\in T_\xi\mathbb{R}^{n+1}.
\end{equation*}
One can check that at $x\in\Sigma$, $\nu_F=\nu_F(x)$ is perpendicular to $T_x\Sigma$ w.r.t. $G(\nu_F)$ in the sense
\begin{equation*}
	G(\nu_F)(\nu_F,\nu_F)=1,\ G(\nu_F)(\nu_F,V)=0,\ \forall V\in T_x\Sigma.
\end{equation*}
Then consider the Riemannian metric on $\Sigma$ defined by $\hat{g}(x):=G(\nu_F(x))\big|_{T_x\Sigma}$, denote the covariant derivative of $\hat{g}$ and $G$ by $\hat{\nabla}$ and $\hat{D}$, the first and second fundamental form of $(\Sigma,\hat{g})\subset(\mathbb{R}^{n+1},G)$ are 
\begin{equation*}
	\hat{g}_{ij}=G(\nu_F)(\partial_ix,\partial_jx),\ \hat{h}_{ij}=G(\nu_F)(\hat{D}_{\partial_i}\nu_F,\partial_jx).
\end{equation*}
It has been proved that the anisotropic principal curvatures $\{\kappa_i^F\}_{i=1,\cdots, n}$ are exactly the eigenvalues of $\left(\hat{h}_i^j\right)=\left((\hat{g}^{-1})^{kj}\hat{h}_{ik}\right)$ and the anisotropic mean curvature w.r.t. $F$ is defined as
\begin{equation*}
	H_F=\text{tr}_{\hat{g}}(\hat{h}).
\end{equation*}

We call $\Sigma$ strictly $F$-mean convex if $H_F\big|_\Sigma>0$, and $F$-mean convex if $H_F\big|_\Sigma\geqslant 0$.

It is easy to check that when $\Sigma=\mathcal{W}$, we have $\nu_F(x)=x$, $\hat{h}_{ij}=\hat{g}_{ij}$ and $H_F=n$.
\subsection{Inverse anisotropic mean curvature flow}\label{sec-IAMCF}$\ $

Let $\Sigma\subset\mathbb{R}^{n+1}$ be a smooth, closed, star-shaped and strictly $F$-mean convex hypersurface given by the embedding $X_0:\mathbb{S}^n\rightarrow\mathbb{R}^{n+1}$. The inverse anisotropic mean curvature flow is the smooth family of embeddings $X:\mathbb{S}^n\times[0,T)\rightarrow\mathbb{R}^{n+1}$ satisfying 
\begin{equation}\label{eq-IAMCF}
	\left\{\begin{aligned}
		\frac{\partial}{\partial t}X(\cdot,t)&=\frac{1}{H_F}\nu_F,\\
		X(\cdot,0)&=X_0(\cdot).
	\end{aligned}\right.
\end{equation}
When $n=1$, the flow becomes a family of smooth embedded closed curves in the plane $\mathbb{R}^{2}$ and the $F$-mean convexity becomes the standard convexity for curves.

If a solution of \eqref{eq-IAMCF} exists, the following evolution equations were obtained in \cite{BAndrews01,CXia17}:
\begin{lem}\label{lem-evl}
	Along the anisotropic curvature flow which evolves as $\frac{\partial}{\partial t}X=\eta\nu_F$, the anisotropic area form $d\mu_F$ and the anisotropic perimeter of $\Sigma_t$ evolve under
	\begin{equation*}
		\partial_t d\mu_F=\eta H_Fd\mu_F\left(=d\mu_F\ \text{if }\eta=\frac{1}{H_F}\right),
	\end{equation*}
\begin{equation*}
	\frac{d}{dt}\big|\Sigma_t\big|_F=\int_{\Sigma}\eta H_Fd\mu_F\left(=\big|\Sigma_t\big|_F\ \text{if }\eta=\frac{1}{H_F}\right).
\end{equation*}
\end{lem}

The smooth existence and convergence results for the inverse anisotropic mean curvature flow starting from a star-shaped and strictly $F$-mean convex hypersurface have been proved by Andrews for $n=1$ and Xia for $n\geqslant 2$. 
\begin{thm}[\cite{BAndrews98,CXia17}]\label{thm-IAMCF}
	Suppose $\Sigma_0$ is smooth, closed, star-shaped and strictly $F$-mean convex. Then there exists a unique, smooth, strictly $F$-mean convex solution $X(\cdot,t)$ to \eqref{eq-IAMCF} for $t\in[0,\infty)$. And the rescaled hypersurfaces $e^{-\frac{t}{n}}X(\cdot,t)$ converge exponentially fast to $\alpha_0\mathcal{W}$ in the $C^\infty$ topology, where $\alpha_0=\big|\Sigma_0\big|_F$.
\end{thm}
\section{Proof of Anisotropic Weighted Isoperimetric Inequalities}\label{sec-proof}
In this section, we prove Theorem \ref{thm-iso-ineq1} and Theorem \ref{thm-iso-ineq2}. 

Let $Q(\Sigma)$ be the functional of smooth, star-shaped hypersurfaces $\Sigma=\partial\Omega$ in $\mathbb{R}^{n+1}$, defined by
\begin{equation}\label{def-Q(Sigma)}
	Q(\Sigma)=\big|\Sigma\big|_F^{-1-\frac{1}{n}}\left(\int_{\Sigma}F^0(x-P)d\mu_F-\text{Vol}(\Omega)\right).
\end{equation}
And for our purpose, let $\Sigma_t$ be the solution of the inverse anisotropic mean curvature flow \eqref{eq-IAMCF}, we will consider the modified flow
\begin{equation*}
	\hat{\Sigma}_t=\partial\hat{\Omega}_t:=e^{-\frac{t}{n}}\Sigma_t+\left(1-e^{-\frac{t}{n}}\right)P.
\end{equation*}
\begin{prop}\label{prop-mono}
	Let $\Sigma_0=\Sigma$ be a smooth, closed, star-shaped and strictly $F$-mean convex hypersurface, $P\in\mathrm{int}(\Omega)$. Then the quantity $Q(\hat{\Sigma}_t)$ is monotonically nonincreasing along the flow. Moreover, $\frac{d}{dt}Q(\hat{\Sigma}_t)=0$ if and only if $\Sigma_t$ is a rescaling of the Wulff shape $\mathcal{W}$ centered at $P$.
\end{prop}
\begin{proof}
It is easy to see that $Q$ is invariant under rescaling in the sense that $Q(kM+P)=Q(M+P)$ for any hypersurface $M$. Let $\Sigma_t=\partial\Omega_t$. 

First we assume that $P\in\mathrm{int}(\Omega_0)$. Using maximal principle, one can show that $P$ stays in the evolving region $\mathrm{int}(\Omega_t)$, so $F^0(x-P)$ is differentiable at $x\in\Sigma_t$. By Lemma \ref{lem-evl}, we have
\begin{equation}\label{eq-mono}
	\begin{aligned}
		&\frac{d}{dt}Q(\hat{\Sigma}_t)=\frac{d}{dt}Q(\Sigma_t)\\
		=&-(1+\frac{1}{n})\big|\Sigma_t\big|_F^{-1-\frac{1}{n}}\left(\int_{\Sigma_t}F^0(x-P)d\mu_F-\text{Vol}(\Omega_t)\right)\\
		&+\big|\Sigma_t\big|_F^{-1-\frac{1}{n}}\left(\int_{\Sigma_t}\left(DF^0(x-P)\cdot x_t+F^0(x-P)\right)d\mu_F-\int_{\Sigma_t}\frac{1}{H_F}\nu_F \cdot \nu d\mu\right)\\
		=&\big|\Sigma_t\big|_F^{-1-\frac{1}{n}}\left(-\frac{1}{n}\int_{\Sigma_t}F^0(x-P)d\mu_F+(1+\frac{1}{n})\text{Vol}(\Omega_t)+\int_{\Sigma_t}\left(DF^0(x-P)\cdot\nu_F-1\right)\frac{1}{H_F}d\mu_F\right).
	\end{aligned}
\end{equation}
By \eqref{eq-F(DF^0)} and the anisotropic Cauchy-Schwarz inequality, we see
\begin{equation}\label{ineq-CW1}
	DF^0(x-P)\cdot\nu_F-1\leq F(DF^0(x-P))F^0(DF(\nu))-1=0,
\end{equation}
\begin{equation}\label{ineq-CW2}
	(n+1)\text{Vol}(\Omega_t)=\int_{\Sigma_t}(x-P)\cdot\nu d\mu\leq\int_{\Sigma_t}F^0(x-P)F(\nu)d\mu=\int_{\Sigma_t}F^0(x-P)d\mu_F.
\end{equation}
Hence $\frac{d}{dt}Q(\hat{\Sigma}_t)\leq 0$. 

When $P\in\overline{\Omega_0^c}$, then there exists a unique $t_0\in[0,\infty)$ such that $P\in\Sigma_{t_0}$. For $t\neq t_0$, $Q(\Sigma_t)$ is still differentiable with $\frac{d}{dt}Q(\hat{\Sigma}_t)\leq 0$. Since $Q(\hat{\Sigma}_t)=Q(\Sigma_t)$ is continuous, we see $Q(\hat{\Sigma}_t)$ is monotonically nonincreasing.

If $\frac{d}{dt}Q(\hat{\Sigma}_t)=0$, then equality is achieved for those anisotropic Cauchy-Schwartz inequalities appeared in \eqref{ineq-CW1} and \eqref{ineq-CW2}, which implies $x-P=F^0(x-P)DF(\nu)=F^0(x-P)\nu_F$. Taking the covariant derivative induced by $\hat{g}$ on $x=F^0(x)\nu_F$, we see $\hat{\nabla}_iF^0(x-P)\nu_F\in T_x(\Sigma_t)$, $\forall 1\leq i\leq n$, and hence $F^0(x-P)$ is a constant on $\Sigma_t$, which means $\Sigma_t$ is a rescaling of the Wulff shape $\mathcal{W}$ centered at $P$.
\end{proof}
\begin{proof}[Proof of Theorem \ref{thm-iso-ineq1}]
	First we consider the case when $\Sigma$ is strictly $F$-mean convex. By Proposition \ref{prop-mono}, let $\Sigma_t$ be the solution of the inverse anisotropic mean curvature flow \eqref{eq-IAMCF} starting from $\Sigma$, then
	\begin{equation*}
		Q(\Sigma)\geq Q(\hat{\Sigma}_t),\ \forall t\geq0.
	\end{equation*}
By the convergence result in Theorem \ref{thm-IAMCF}, the rescaled flow $\hat{\Sigma}_t$ converges smoothly to a rescaled and translated Wulff shape $\mathcal{W}_{\alpha_0}(P)$. Note that on $\mathcal{W}$ we have $\big|\mathcal{W}\big|_F=\int_{\mathcal{W}}F^0(x)d\mu_F=(n+1)\text{Vol}(L)$, where $\partial L=\mathcal{W}$. Then
\begin{equation}
	Q(\Sigma)\geq\lim_{t\rightarrow\infty}Q(\hat{\Sigma}_t)=Q(\alpha_0\mathcal{W}+P)=Q(\mathcal{W}+P)=(n+1)^{-1-\frac{1}{n}}n\text{Vol}(L)^{-\frac{1}{n}},
\end{equation}
which is the inequality we asserted. By Proposition \ref{prop-mono}, we also see that equality holds if and only if $\Sigma$ is a rescaling of the Wulff shape $\mathcal{W}$ centered at $P$.

In the case when $\Sigma$ is only $F$-mean convex, the inequality follows from the approximation. To prove the equality condition, we use a similar argument in \cite{Guan-Li09}: 

Suppose that $\Sigma$ is a star-shaped and $F$-mean convex hypersurface which attains the equality of \eqref{eq-iso-ineq1}. Let $\Sigma^+=\{x\in\Sigma\ :\ H_F(x)>0\}$. We can find the smallest $\alpha$ such that $\Omega\subset L_\alpha(P)$, then $\mathcal{W}_\alpha(P)$ touches $\Sigma$ at some point $p\in\Sigma^+$ tangentially with $F^0(p-P)=\alpha<\infty$ and hence $\Sigma^+$ is open and non-empty. 

If $P\notin\Sigma^+$, pick any $\eta\in C_0^2(\Sigma^+)$ compactly supported in $\Sigma^+$, let $\Sigma_\varepsilon$ be the smooth family of variational hypersurfaces generated by the vector field $V=\eta\nu_F$, $\Sigma_0=\Sigma$. Then $\Sigma_\varepsilon$ remains to be $F$-mean convex when $|\varepsilon|$ is sufficiently small. We have
\begin{equation*}
	Q(\Sigma_\varepsilon)\geq c,\ Q(\Sigma_0)=c,\ \text{where }c=\frac{n}{(n+1)^{1+\frac{1}{n}}}\text{Vol}(L)^{-\frac{1}{n}}.
\end{equation*} 
Thus by similar calculation in \eqref{eq-mono},
\begin{equation*}
	\begin{aligned}
			0=&\frac{d}{d\varepsilon}\big|_{\varepsilon=0}Q(\Sigma_\varepsilon)\\
			=&\big|\Sigma\big|_F^{-1-\frac{1}{n}}\int_{\Sigma^+}\eta\cdot\left[\frac{H_F}{n\big|\Sigma\big|_F}\left((n+1)\text{Vol}(\Omega)-\int_{\Sigma}F^0(x-P)d\mu_F\right)+\left(DF^0(x-P)\cdot\nu_F-1\right)\right.\\
			&\ \ \ \ \ \ \ \ \ \ \ \ \ \ \ \ \ \ \ \ \left.+H_F\left(F^0(x-P)-\frac{\int_{\Sigma}F^0d\mu_F}{|\Sigma|_F}\right)\right]d\mu_F,\\
			=&\big|\Sigma\big|_F^{-1-\frac{1}{n}}\int_{\Sigma^+}\eta\cdot\left[H_F\left(F^0(x-P)-\frac{\big|\Sigma\big|_F^{\frac{1}{n}}}{\big|\mathcal{W}\big|_F^{\frac{1}{n}}}\right)+\left(DF^0(x-P)\cdot\nu_F-1\right)\right]d\mu_F,
	\end{aligned}
\end{equation*}
where we used that equality of \eqref{eq-iso-ineq1} is achieved on $\Sigma$ and $(n+1)\text{Vol}(L)=\big|\mathcal{W}\big|_F$ in the last equality. Since $\eta$ is arbitrary, we have
\begin{equation*}
	H_F\left(F^0(x-P)-\frac{\big|\Sigma\big|_F^{\frac{1}{n}}}{\big|\mathcal{W}\big|_F^{\frac{1}{n}}}\right)+\left(DF^0(x-P)\cdot\nu_F-1\right)=0,\ \ \forall\ x\in\Sigma^+.
\end{equation*} 
Recall that $p\in\Sigma^+$ is the point where $\Sigma$ touches $F^0(p-P)\mathcal{W}+P$ tangentially, so we have $DF^0(x-P)\cdot\nu_F\big|_{x=p}=1$, then 
\begin{equation*}
	F^0(p-P)=\frac{\big|\Sigma\big|_F^{\frac{1}{n}}}{\big|\mathcal{W}\big|_F^{\frac{1}{n}}},
\end{equation*}
hence 
\begin{equation*}
	\begin{aligned}
			&H_F\left(F^0(x-P)-F^0(p-P)\right)\\
			\geq&H_F\left(F^0(x-P)-\frac{\big|\Sigma\big|_F^{\frac{1}{n}}}{\big|\mathcal{W}\big|_F^{\frac{1}{n}}}\right)+\left(DF^0(x-P)\cdot\nu_F-1\right)=0,\ \forall\ x\in\Sigma^+.
	\end{aligned}
\end{equation*}
Since $H_F>0$ and $F^0(x-P)\leq F^0(p-P)$ on $\Sigma^+$, we see $F^0(x-P)\equiv F^0(p-P)=\alpha$ on $\Sigma^+$. 

If $P\in\Sigma^+$, we can pick $\eta\in C_0^2(\Sigma^+\backslash \{P\})$ and get $F^0(x-P)\equiv\alpha$ on $\Sigma^+\backslash \{P\}$. Since $\Sigma^+$ is open, by continuity we also have $F^0(x-P)\equiv\alpha$ on $\Sigma^+$. 

Hence $\Sigma^+=\Sigma\cap\mathcal{W}_\alpha(P)$ with the anisotropic mean curvature positively bounded from below by $\frac{n}{\alpha}$. We conclude that $\Sigma^+$ is closed. Therefore $\Sigma=\Sigma^+=\{x\in\mathbb{R}^{n+1}:F^0(x-P)=\alpha\}=\mathcal{W}_\alpha(P)$.
\end{proof}
\begin{proof}[Proof of Theorem \ref{thm-iso-ineq2}]
	For $p\geq1$, we can use H\"older's inequality and Theorem \ref{thm-iso-ineq1} to get
	\begin{equation}
		\left(\int_{\Sigma}\left(F^0(x-P)\right)^pd\mu_F\right)^{\frac{1}{p}}\left(\big|\Sigma\big|_F\right)^{1-\frac{1}{p}}\geq\int_{\Sigma}F^0(x-P)d\mu_F\geq c\big|\Sigma\big|_F^{1+\frac{1}{n}}+\text{Vol}(\Omega),
	\end{equation}
where $c=\frac{n}{(n+1)^{1+\frac{1}{n}}}\text{Vol}(L)^{-\frac{1}{n}}$. Since $na+b\geq(n+1)a^{\frac{n}{n+1}}b^{\frac{1}{n+1}},\ \forall\ a,b\geq0$, we see
\begin{equation}
		\frac{\left(\int_{\Sigma}\left(F^0(x-P)\right)^pd\mu_F\right)^{\frac{1}{p}}}{\big|\Sigma\big|_F^{\frac{1}{p}}\text{Vol}(\Omega)^{\frac{1}{n+1}}}\geq c \frac{\big|\Sigma\big|_F^{\frac{1}{n}}}{\text{Vol}(\Omega)^{\frac{1}{n+1}}}+\frac{\text{Vol}(\Omega)^{\frac{n}{n+1}}}{\big|\Sigma\big|_F}\geq\frac{(n+1)}{n^{\frac{n}{n+1}}}c^{\frac{n}{n+1}}=\text{Vol}(L)^{-\frac{1}{n+1}},
\end{equation}
which is the inequality for $p$-momentum. By the equality condition of Theorem \ref{thm-iso-ineq1}, we see that \eqref{ineq-aniso2'} attains equality only for the rescalings of the Wulff shape.
\end{proof}

\section{Proof of Stability Results}\label{sec-proof-stability}
In this section, we use the inverse anisotropic mean curvature flow to prove Theorem \ref{thm-L1-stability}, \ref{thm-aniso-stability} and Corollary \ref{cor-weinstock-stability}. We argue by showing that after a proper period of time, the modified flow hypersurface will become close to both the Wulff shape and the initial hypersurface (in $L^1$ or Hausdorff distance sense).   
\begin{proof}[Proof of \eqref{L1-stable1} in Theorem \ref{thm-L1-stability}]$\ $

	\textbf{Step 1.} Define $Q(\Sigma)$ as in (\ref{def-Q(Sigma)}). Denote $\Sigma_0=\Sigma$, $\Sigma_t$ as the solution of inverse anisotropic mean curvature flow (\ref{eq-IAMCF})
	and $\Omega_t$ as the domain enclosed by $\Sigma_t$. 
	Let
	\begin{equation}
		\varepsilon := Q(\Sigma)-Q(\mathcal{W}+P)=Q(\Sigma)-(n+1)^{-(1+\frac{1}{n})}n\mathrm{Vol}(L)^{-\frac{1}{n}}. \label{eps}
	\end{equation}
	Consider the modified flow hypersurface $\hat{\Sigma}_t=\mathrm{e}^{-\frac{t}{n}}\Sigma_t+\left(1-e^{-\frac{t}{n}}\right)P$ and the domain $\hat{\Omega}_t$ enclosed by $\hat{\Sigma}_t$. Along the flow (\ref{eq-IAMCF}),  $\{\Sigma_t\}$ satisfy the evolution equation (\ref{eq-mono}) and by inequality (\ref{ineq-CW1}),
	\begin{eqnarray}
		\dfrac{\mathrm{d}}{\mathrm{d}t}Q(\Sigma_t) &\leqslant& \big|\Sigma_t\big|_F^{-(1+\frac{1}{n})} \left( -\dfrac{1}{n}\int_{\Sigma_t} F^0(x-P) \mathrm{d}\mu_F + \dfrac{n+1}{n} \mathrm{Vol}(\Omega_t) \right) \notag\\
		&=&-\dfrac{1}{n}\big|\Sigma_t\big|_F^{-(1+\frac{1}{n})} \int_{\Sigma_t} \left( F^0(x-P)F(\nu)-(x-P)\cdot\nu  \right) \mathrm{d}\mu \notag\\
		&=& -\dfrac{1}{n}\big|\Sigma\big|_F^{-(1+\frac{1}{n})}\int_{\hat{\Sigma}_t} \left( F^0(x-P) F(\nu)-(x-P)\cdot\nu \right) \mathrm{d}\mu\label{dtq},
	\end{eqnarray}
	where in (\ref{dtq}) we used $\big|\Sigma_t\big|_F = \mathrm{e}^{t}\big|\Sigma\big|_F$ and the scaling invariance property $Q(kM+P)=Q(M+P)$ for any hypersurface $M$. By Theorem \ref{thm-IAMCF}, as $t\to\infty$, there holds
	\begin{equation}
		\hat{\Sigma}_t\to \alpha\mathcal{W}+P=\mathcal{W}_\alpha(P)\ \ \text{in}\ C^\infty.
	\end{equation}
Hence
	\begin{equation}
		\int_0^{+\infty}\dfrac{\mathrm{d}}{\mathrm{d}t}Q(\Sigma_t) \mathrm{d}t = Q(\mathcal{W}+P)-Q(\Sigma) = -\varepsilon. \label{de}
	\end{equation}
	Integrate (\ref{dtq}) along $[0,+\infty)$ on both sides, we get
	\begin{equation}
		\int_0^{+\infty} \dfrac{\mathrm{d}}{\mathrm{d}t}Q(\Sigma_t) \mathrm{d}t \leqslant -\dfrac{1}{n}\big|\Sigma\big|_F^{-(1+\frac{1}{n})}\int_0^{+\infty}\int_{\hat{\Sigma}_t} \left( F^0(x-P)F(\nu) -(x-P)\cdot \nu \right) \mathrm{d}\mu\mathrm{d}t. \label{tsig} 
	\end{equation}
	Note that the initial domain satisfies $B_{\rho_{-}}(P)\subset\Omega\subset B_{\rho_{+}}(P)$, then there exists $c=c(\rho_{+},F)\geq c'=c'(\rho_{-},F)>0$, such that $L_{c'}(P)\subset\Omega\subset L_c(P)$. By the standard maximum principle, this property is preserved along the inverse anisotropic mean curvature flow, i.e., $L_{e^{\frac{t}{n}}c'}(P)\subset\Omega_t\subset L_{e^{\frac{t}{n}}c}(P)$. Hence 
	\begin{equation}\label{avoidence-barrier}
		L_{c'}(P)\subset\hat{\Omega}_t\subset L_c(P),\ \ c'\leq F^0(x-P)\leq c\ \ \text{on}\ \hat{\Sigma}_t,
	\end{equation}
and then
	\begin{equation}\label{area-upper-bound}
		\big|\Sigma\big|_F=\big|\hat{\Sigma}_\infty\big|_F=\big|\mathcal{W}_\alpha(P)\big|_F\leq\big|\mathcal{W}_c(P)\big|_F\leq C(n,F,\rho_{+}).
	\end{equation} 
 Combining (\ref{de}), (\ref{tsig}), \eqref{avoidence-barrier} and \eqref{area-upper-bound}, we conclude that
	\begin{equation}\label{infty-integral}
		\begin{aligned}
	&\int_0^{\sqrt{\varepsilon}} \int_{\hat{\Sigma}_t}  \left(F(\nu) -\frac{(x-P)}{F^0(x-P)}\cdot \nu \right) \mathrm{d}\mu\mathrm{d}t\\
\leq&\int_0^{+\infty}\int_{\hat{\Sigma}_t} \left(F(\nu) -\frac{(x-P)}{F^0(x-P)}\cdot \nu \right) \mathrm{d}\mu\mathrm{d}t\\
\leq&C(n,F,\rho_{-},\rho_{+})\varepsilon,		
		\end{aligned}
	\end{equation}
then there exists $t_{\varepsilon}\in(0,\sqrt{\varepsilon})$ such that
	\begin{equation}
		\int_{\hat{\Sigma}_{t_{\varepsilon}}}  \left(F(\nu) -\frac{(x-P)}{F^0(x-P)}\cdot \nu \right) \mathrm{d}\mu \leqslant C(n,F,\rho_{-},\rho_{+})\sqrt{\varepsilon}. \label{lowerdist}
	\end{equation}
	
	\textbf{Step 2.} We will use \eqref{lowerdist} to show that $\hat{\Sigma}_{t_{\varepsilon}}$ is also $L_1$ close to a Wulff shape. By the divergence Theorem,
\begin{equation*}
	\begin{aligned}
			\int_{\hat{\Sigma}_{t_{\varepsilon}}}\frac{x-P}{F^0(x-P)}\cdot\nu\mathrm{d}\mu=&\int_{\hat{\Omega}_{t_\epsilon}}\mathrm{div}\left(\frac{x-P}{F^0(x-P)}\right)\mathrm{d}x\\
			=&\int_{\hat{\Omega}_{t_\epsilon}}\left(\frac{\mathrm{div}(x)}{F^0(x-P)}-\frac{DF^0(x-P)\cdot(x-P)}{\left(F^0(x-P)\right)^2}\right)\mathrm{d}x\\
			=&\int_{\hat{\Omega}_{t_\epsilon}}\frac{n}{F^0(x-P)} \mathrm{d}x,
	\end{aligned}
\end{equation*} 
hence
\begin{equation}\label{L1-est1}
	\big|\hat{\Sigma}_{t_\epsilon}\big|_F-\int_{\hat{\Omega}_{t_\epsilon}}\frac{n}{F^0(x-P)}\mathrm{d}x\leq C(n,F,\rho_{-},\rho_{+})\sqrt{\epsilon}.
\end{equation}
Take $a>0$ such that $\mathrm{Vol}(\hat{\Omega}_{t_\epsilon})=\mathrm{Vol}(L_a(P))$, then we have
\begin{equation}\label{L1-est2}
	\begin{aligned}
		&\int_{L_a(P)}\frac{1}{F^0(x-P)}\mathrm{d}x-\int_{\hat{\Omega}_{t_\epsilon}}\frac{1}{F^0(x-P)}\mathrm{d}x\\=&\int_{L_a(P)\backslash\hat{\Omega}_{t_\epsilon}}\frac{1}{F^0(x-P)}\mathrm{d}x-\int_{\hat{\Omega}_{t_\epsilon}\backslash L_a(P)}\frac{1}{F^0(x-P)}\mathrm{d}x\\
		\geq&\int_{L_a(P)\backslash\hat{\Omega}_{t_\epsilon}}\frac{1}{a}\mathrm{d}x-\int_{\hat{\Omega}_{t_\epsilon}\backslash L_a(P)}\frac{1}{a}\mathrm{d}x=0.
	\end{aligned}
\end{equation}
And by the coarea formula,
\begin{equation}\label{L1-est3}
	\int_{L_a(P)}\frac{1}{F^0(x-P)}\mathrm{d}x=\int_{0}^{r}\frac{\big|\mathcal{W}_s(P)\big|_F}{s}\mathrm{d}s=\frac{1}{n}\big|\mathcal{W}_r(P)\big|_F.
\end{equation}
Then  by the quantitative Wulff inequality (\ref{equ-quantiwulff}), combining \eqref{L1-est1}, \eqref{L1-est2}, \eqref{L1-est3}, we see that $\hat{\Sigma}_{t_\epsilon}$ is $L_1$ close to a Wulff shape, in the sense that
\begin{equation}\label{t_epsilon-asymmetry}
	\min_{P'\in\mathbb{R}^{n+1}}\big|\hat{\Omega}_{t_\epsilon}\triangle L_a(P')\big|^2\leq C(n,F)\left(\big|\hat{\Sigma}_{t_\epsilon}\big|_F-\big|\mathcal{W}_a(P)\big|_F\right)\leq C(n,F,\rho_{-},\rho_{+})\sqrt{\epsilon}.
\end{equation}

\textbf{Step 3.} Let $P'$ be the point which achieves $\min\limits_{P'\in\mathbb{R}^{n+1}}\big|\hat{\Omega}_{t_\epsilon}\triangle L_a(P')\big|$. Take $b>0$ such that $\mathrm{Vol}(L_b(P'))=\mathrm{Vol}(\Omega)$. Then by the uniform estimates of the positive lower bound of $H_F$ along the inverse anisotropic mean curvature flow (see e.g., in \cite{CXia17}),
\begin{equation}\label{t_epsilon-asymmetry2}
	\begin{aligned}
		\big|L_a(P')\triangle L_b(P')\big|=&\big|\mathrm{Vol}(L_a(P'))-\mathrm{Vol}(L_b(P'))\big|=\big|\mathrm{Vol}(\hat{\Omega}_{t_\epsilon})-\mathrm{Vol}(\Omega)\big|\\
		\leq&t_\epsilon\cdot\max_{\tau\in[0,t_\epsilon]}\bigg|\frac{\mathrm{d}}{\mathrm{d}\tau}\mathrm{Vol}(\hat{\Omega}_\tau)\bigg|\\
		\leq&\sqrt{\epsilon}\max_\tau\bigg|-\frac{n+1}{n}\mathrm{Vol}(\hat{\Omega}_\tau)+\int_{\hat{\Sigma}_\tau}\frac{1}{H_F}\mathrm{d}\mu_F\bigg|\\
		\leq& C(n,F,\rho_{+},\min\limits_{\Sigma} H_F)\sqrt{\epsilon}.
	\end{aligned}
\end{equation}
For each $t>0$, $\hat{\Sigma}_t$ is star-shaped w.r.t. $P(t):=\left(1-e^{\frac{t}{n}}\right)P$. Then in a small neighborhood $\tau\in[t-\eta,t+\eta]$, we can also assume $\hat{\Sigma}_\tau$ to be star-shaped w.r.t. $P(t)$ and write them as $\hat{\Sigma}_\tau=\{r(\xi,\tau)\cdot\theta+P(t):\theta\in\mathbb{S}^n\}$. Then
\begin{equation}
	\begin{aligned}
		\bigg|\frac{\mathrm{d}}{\mathrm{d}\tau}\big|\Omega\triangle\hat{\Omega}_{\tau}\big|\bigg|=&\bigg|\frac{\mathrm{d}}{\mathrm{d}\tau}\int_{\mathbb{S}^n}\bigg|\int_0^{r(\xi,\tau)}s^n\mathrm{d}s-\int\limits_{\{s:s\theta+P(t)\in\Omega\}}s^n\mathrm{d}s\bigg|\mathrm{d}\mathbb{S}^n\bigg|\\
		\leq&\int_{\mathbb{S}^n}\bigg|\frac{\mathrm{d}}{\mathrm{d}\tau}r\bigg|r^n\mathrm{d}\mathbb{S}^n=\int_{\hat{\Sigma}_\tau}\big|\frac{\mathrm{d}}{\mathrm{d}\tau}\hat{X}(\xi,\tau)\cdot\nu\big|r^n\mathrm{d}\mu\\
		\leq&C(n,F,\rho_{+},\min\limits_{\Sigma} H_F),
	\end{aligned}
\end{equation}
where $\hat{X}(\xi,t)=e^{-\frac{t}{n}}X(\xi,t)+\left(1-e^{-\frac{t}{n}}\right)P$ and we used the variation formula of star-shaped hypersurfaces written as graphs on $\mathbb{S}^n$ in the equality in the second line. Hence for each $t>0$, we have $\bigg|\frac{\mathrm{d}}{\mathrm{d}t}\big|\Omega\triangle\hat{\Omega}_{t}\big|\bigg|\leq C(n,F,\rho_{+},\min\limits_{\Sigma} H_F)$ and
\begin{equation}\label{t_epsilon-asymmetry3}
	\big|\Omega\triangle\hat{\Omega}_{t_\epsilon}\big|\leq t_\epsilon\cdot\max_{t\in[0,t_\epsilon]}\bigg|\frac{\mathrm{d}}{\mathrm{d}t}\big|\Omega\triangle\hat{\Omega}_{t}\big|\bigg|\leq C(n,F,\rho_{+},\min\limits_{\Sigma} H_F)\sqrt{\epsilon}.
\end{equation}
Then by \eqref{t_epsilon-asymmetry}, \eqref{t_epsilon-asymmetry2} and \eqref{t_epsilon-asymmetry3},
\begin{equation}
	\begin{aligned}
		\alpha_F(\Omega)\leq&\frac{\big|\Omega\triangle L_b(P')\big|}{\mathrm{Vol}(\Omega)}\\
		\leq& C(n,F,\rho_{-})\left(\big|\Omega\triangle\hat{\Omega}_{t_\epsilon}\big|+\big|\hat{\Omega}_{t_\epsilon}\triangle L_a(P')\big|+\big|L_a(P')\triangle L_b(P')\big|\right)\\
		\leq&C(n,F,\rho_{-},\rho_{+},\min\limits_{\Sigma} H_F)\left(\epsilon^{\frac{1}{4}}+\sqrt{\epsilon}\right).
	\end{aligned}
\end{equation}
\end{proof}

\begin{proof}[Proof of \eqref{L1-fostable} in Theorem \ref{thm-L1-stability}]
We have
    \begin{eqnarray}
    	\dfrac{\displaystyle\int_{\Sigma}\left(F^0(x-P)\right)^p\mathrm{d}\mu_F}{|\Sigma|_F\mathrm{Vol}(\Omega)^{\frac{p}{n+1}}}-\mathrm{Vol}(L)^{-\frac{p}{n+1}} &\geqslant&  \dfrac{\left( \displaystyle\int_{\Sigma} F^0(x-P) \mathrm{d}\mu_F \right)^p}{|\Sigma|_F^p \mathrm{Vol}(\Omega)^{\frac{p}{n+1}}} - \mathrm{Vol}(L)^{-\frac{p}{n+1}} \notag\\
    	&\geqslant& C \left( \dfrac{\displaystyle\int_{\Sigma} F^0(x-P) \mathrm{d}\mu_F}{|\Sigma|_F\mathrm{Vol}(\Omega)^{\frac{1}{n+1}}} -\mathrm{Vol}(L)^{-\frac{1}{n+1}}\right), \label{pmo}
    \end{eqnarray}
    where we used H\"older's inequality in the first inequality, and \eqref{avoidence-barrier}, \eqref{area-upper-bound} in the second inequality.
    Since $f$ is monotonously increasing and that the lower bound of $|\Sigma|_F$ is controlled by $\rho_{-}$, we reorganize (\ref{stable1}) to obtain
    \begin{equation}
    	\dfrac{|\mathcal{W}|_F^{\frac{1}{n}}}{|\Sigma|_F^{1+\frac{1}{n}}}\left( \int_{\Sigma} F^0(x-P) \mathrm{d}\mu_F - \mathrm{Vol}(\Omega) \right) -\dfrac{n}{n+1} \geqslant Cf_1^{-1}\left(\frac{\alpha_F(\Omega)}{C}\right).
    \end{equation}
    Therefore,
    \begin{equation}\label{1mo}
    	\begin{aligned}
    		  \dfrac{\displaystyle\int_{\Sigma} F^0(x-P) \mathrm{d}\mu_F}{|\Sigma|_F\mathrm{Vol}(\Omega)^{\frac{1}{n+1}}}=& \dfrac{\displaystyle\int_{\Sigma} F^0(x-P) \mathrm{d}\mu_F-\mathrm{Vol}(\Omega)}{|\Sigma|_F\mathrm{Vol}(\Omega)^{\frac{1}{n+1}}} + \dfrac{\mathrm{Vol}(\Omega)^{\frac{n}{n+1}}}{|\Sigma|_F}  \notag\\
    		\geqslant& \dfrac{|\Sigma|_F^{\frac{1}{n}}}{\mathrm{Vol}(\Omega)^{\frac{1}{n+1}}|\mathcal{W}|_F^{\frac{1}{n}}}\left( Cf_1^{-1}\left(\frac{\alpha_F(\Omega)}{C}\right)+\dfrac{n}{n+1} \right) +  \dfrac{\mathrm{Vol}(\Omega)^{\frac{n}{n+1}}}{|\Sigma|_F}. 
    	\end{aligned}
    \end{equation}
   Then by Young's inequality, (\ref{pmo}) and (\ref{1mo}), we get 
   \begin{equation*}
   		\dfrac{\displaystyle\int_{\Sigma}\left(F^0(x-P)\right)^p\mathrm{d}\mu_F}{|\Sigma|_F\mathrm{Vol}(\Omega)^{\frac{p}{n+1}}}-\mathrm{Vol}(L)^{-\frac{p}{n+1}}\geq Cf_1^{-1}\left(\frac{\alpha_F(\Omega)}{C}\right).
   \end{equation*}
Using the fact that $f_1(ab)\geq \min\{\sqrt{a},a^{ \frac{1}{4}} \}f_1(b),\ a,b\geq0$, we get the conclusion. 
\end{proof}

If we assume further that the star-shaped hypersurface $\Sigma$ has a $C^1$ bound on its radial function $r$, then by the Poincar\'e inequality we can obtain the Hausdorff distance estimate between $\Sigma$ and a proper Wulff shape.
\begin{proof}[Proof of \eqref{stable1} in Theorem \ref{thm-aniso-stability}]
	Taking $P$ as the origin in the proof of \eqref{L1-stable1}. We can use the gradient bound to obtain the upper bound of $|\Sigma|_F$. Then by the same argument, there exists $t_{\varepsilon}\in(0,\sqrt{\varepsilon})$ such that
	\begin{equation}
		\int_{\hat{\Sigma}_{t_{\varepsilon}}} \left( F^0(x)F(\nu)-x\cdot \nu \right) \mathrm{d}\mu \leqslant C(n,F,\min\limits_{\mathbb{S}^n}r,\max\limits_{\mathbb{S}^n}\left|\nabla_{\mathbb{S}^n}r\right|)\sqrt{\varepsilon}, \label{lowerdist1}
	\end{equation}
	where $\epsilon=Q(\Sigma)-(n+1)^{-(1+\frac{1}{n})}n\mathrm{Vol}(L)^{-\frac{1}{n}}$. 
	
	For any $(\xi,t)\in\mathbb{S}^n\times\mathbb{R}^+$, let $\gamma(\tau):[0,t]\rightarrow\mathbb{R}^{n+1}$ be the curve starting from $X(\xi,0)$ to $e^{-\frac{t}{n}}X(\xi,t)$ generated by the rescaled flow. Then
	\begin{equation*}
		\begin{aligned}
			\mathrm{dist}_{\mathbb{R}^{n+1}}(X(\xi,0)&,e^{-\frac{t}{n}}X(\xi,t))\leq t\max\limits_{[0,t]}\bigg|\partial_{\tau}\left(\mathrm{e}^{-\frac{\tau}{n}}X(\xi,\tau)\right)\bigg|\\
			&=t\bigg| -\dfrac{1}{n}\hat{X}(\xi,\tau)+\dfrac{1}{\hat{H}_F} \nu_F(\xi,\tau)\bigg|\\
			&\leq C(n,F,\min\limits_{\mathbb{S}^n}r,\max\limits_{\mathbb{S}^n}\left|\nabla_{\mathbb{S}^n}r\right|,\min\limits_{\Sigma} H_F)t,
		\end{aligned}
	\end{equation*}
	where we used the fact again that the $C^1$ bound of $\hat{\Sigma}_t$ and $\min\limits_{\hat{\Sigma}_t} H_F$ are uniformly controlled by the initial datum on $\Sigma$. Hence the Hausdorff distance between $\Sigma$ and $\hat{\Sigma}_t$ satisfies
	\begin{equation}
		\mathrm{dist}_F(\Sigma,\hat{\Sigma}_{t_{\varepsilon}})\leqslant C(n,F,\min\limits_{\mathbb{S}^n}r,\max\limits_{\mathbb{S}^n}\left|\nabla_{\mathbb{S}^n} r \right|,\min\limits_{\Sigma} H_F)\sqrt{\varepsilon}.  \label{dist1}
	\end{equation}
	
	In the following, we use \eqref{lowerdist1} to show that $\hat{\Sigma}_{t_{\varepsilon}}$ is close to a rescaled Wulff shape. Since $\hat{\Sigma}_{t_{\varepsilon}}$ and the Wulff shape $\mathcal{W}$ are both star-shaped, we represent them as the graphs on the sphere $\mathbb{S}^n$. Denote
	\begin{eqnarray*}
		\hat{\Sigma}_{t_{\varepsilon}} &=& \{(r(\theta),\theta):\theta\in\mathbb{S}^n\},\\
		\mathcal{W} &=&\{(\rho(\theta),\theta):\theta\in\mathbb{S}^n\},
	\end{eqnarray*}
	then by the definition of Wulff shape, we know
	\begin{equation}
		F^0(\theta)\rho(\theta)=1. \label{therho}
	\end{equation}
	Here we extend $\rho:\mathbb{S}^n \to \mathbb{R}^+$ homogeneously in $\mathbb{R}^{n+1}$ by
	\begin{equation}
		\rho:\mathbb{R}^n\backslash\{0\} \to \mathbb{R}^+,\ \ \ \ \rho(x)=\rho\left(\dfrac{x}{|x|}\right)|x|.
	\end{equation} 
	We denote $\nabla_{\mathbb{S}^n}$ be the Levi-Civita connection on $\mathbb{S}^n$, then the unit normal vector $\nu$ at $x\in\hat{\Sigma}_{t_{\varepsilon}}$ can be expressed as
	\begin{equation}\label{eq-nu}
		\nu = \dfrac{\theta-\dfrac{\nabla_{\mathbb{S}^n}r}{r}}{\left| \theta-\dfrac{\nabla_{\mathbb{S}^n}r}{r} \right|},\ \ \theta\in\mathbb{S}^n.
	\end{equation}
	It follows that
	\begin{eqnarray}
		F^0(x)F(\nu)-x\cdot \nu &=& \frac{r}{\rho}\cdot \frac{F\left(\theta-\dfrac{\nabla_{\mathbb{S}^n}r}{r}\right)}{\left| \theta-\dfrac{\nabla_{\mathbb{S}^n}r}{r} \right|}-x\cdot\nu\notag\\
		&=& \left(x\cdot\nu\right)\left[F\left(\frac{\theta}{\rho}-\frac{\nabla_{\mathbb{S}^n}r}{r\rho}\right)-1\right], \label{radial}
	\end{eqnarray}
	where (\ref{therho}), \eqref{eq-nu} and the 1-homogeneous property of $F^0$ and $F$ are used. We also identify $\theta$ with $\dfrac{\partial}{\partial r}$. Note that $\forall\ x\in\mathbb{R}^{n+1}\backslash\{0\}$,
	\begin{equation}
		F^0(x) = F^0\left( \dfrac{x}{|x|} \right)|x| = \dfrac{1}{\rho\left( \dfrac{x}{|x|} \right)} |x| = \dfrac{|x|^2}{\rho(x)},
	\end{equation}
	then
	\begin{equation}
		DF^0(x) = -\dfrac{D\rho(x)}{\rho^2(x)} |x|^2 + \dfrac{2x}{\rho(x)},
	\end{equation}
	thus using the radial parametrization on $\mathcal{W}$, we get
	\begin{equation}
		DF^0(\theta) = -\dfrac{D\rho(\theta)}{\rho^2(\theta)} + \dfrac{2\theta}{\rho(\theta)} = -\dfrac{\rho\theta+\nabla_{\mathbb{S}^n}\rho}{\rho^2} + \dfrac{2\theta}{\rho} = \dfrac{\theta}{\rho} - \dfrac{\nabla_{\mathbb{S}_n}\rho}{\rho^2}. \label{df}
	\end{equation}
	After combining (\ref{radial}), (\ref{df}), and using
	$$\dfrac{\nabla_{\mathbb{S}^n}\rho}{\rho^2} -\dfrac{\nabla_{\mathbb{S}^n}r}{\rho r} = - \dfrac{1}{r}\left( \nabla_{\mathbb{S}^n} \dfrac{r}{\rho} \right),$$
	we arrive at
	\begin{eqnarray} 
		F^0(x)F(\nu)-x\cdot \nu  &=&\left(x\cdot\nu\right)\left[F\left(DF^0(\theta) + \dfrac{\nabla_{\mathbb{S}^n}\rho}{\rho^2} -\dfrac{\nabla_{\mathbb{S}^n}r}{\rho r}\right)-1\right] \notag \\
		&=&\left(x\cdot\nu\right)\left[F\left( DF^0(\theta) - \dfrac{1}{r}\left( \nabla_{\mathbb{S}^n} \dfrac{r}{\rho} \right)\right)-1\right]. \label{domi}
	\end{eqnarray}
	Consider $\mathcal{W}^*$ to be the dual of $\mathcal{W}$, that is
	\begin{equation}
		\mathcal{W}^* := \{x\in\mathbb{R}^{n+1}:F(x)=1\}.
	\end{equation}
	Note that, by (\ref{eq-DF(DF^0)}), the normal vector of $\mathcal{W}^*$ at $DF^0(\theta)\in \mathcal{W}^*$ is $DF\left( DF^0(\theta) \right)=\dfrac{\theta}{F^0(\theta)}$, then
	$$-\dfrac{1}{r}\left( \nabla_{\mathbb{S}^n} \dfrac{r}{\rho} \right)\in T_{\theta}\mathbb{S}^n = T_{DF^0(\theta)}\mathcal{W}^*.$$
	Again, by the fact that the $C^1$ bound of $\hat{\Sigma}_t$ is uniformly controlled by the initial datum on $\Sigma$, and since $\mathcal{W}^*$ is strictly convex, we can expand $F\left( DF^0(\theta) - \dfrac{1}{r}\left( \nabla_{\mathbb{S}^n} \dfrac{r}{\rho} \right) \right)$ at $DF^0(\theta)\in\mathcal{W}^*$ and find a positive constant $C=C(n,F,\min\limits_{\mathbb{S}^n}r,\max\limits_{\mathbb{S}^n}|\nabla_{\mathbb{S}^n}r|)$ such that
	\begin{eqnarray}
		F\left( DF^0(\theta) - \dfrac{1}{r}\left( \nabla_{\mathbb{S}^n} \dfrac{r}{\rho} \right) \right) \geqslant 1+ C\dfrac{1}{r^2}\left| \nabla_{\mathbb{S}^n} \dfrac{r}{\rho} \right|^2 \geqslant 1 + C\left| \nabla_{\mathbb{S}^n} \dfrac{r}{\rho} \right|^2,
	\end{eqnarray}
	where we used  $F(DF^0(\theta))=1$. Consequently, (\ref{domi}) becomes
	\begin{equation}\label{lower-gradient-expand}
		F^0(x)F(\nu)-x\cdot \nu \geq C(n,F,\min\limits_{\mathbb{S}^n}r,\max\limits_{\mathbb{S}^n}|\nabla_{\mathbb{S}^n}r|)(x\cdot\nu)\left| \nabla_{\mathbb{S}^n} \dfrac{r}{\rho} \right|^2.
	\end{equation}
	Note that
	\begin{equation}
		(x\cdot\nu)\mathrm{d}\mu  = r^{n+1} \mathrm{d}\mathbb{S}^n \geqslant C(n,\min\limits_{\mathbb{S}^n}r) \mathrm{d}\mathbb{S}^n, 
	\end{equation}
	combining \eqref{lowerdist1} and \eqref{lower-gradient-expand}, we get
	\begin{equation}
		\begin{aligned}
			\int_{\mathbb{S}^n} \left|\dfrac{r}{\rho} -a\right|^2 \mathrm{d}\mathbb{S}^n&\leq C(n,F,\min\limits_{\mathbb{S}^n}r,\max\limits_{\mathbb{S}^n}|\nabla_{\mathbb{S}^n}r|)\int_{\mathbb{S}^n} \left| \nabla_{\mathbb{S}^n} \dfrac{r}{\rho} \right|^2 \mathrm{d}\mathbb{S}^n\\
			&\leq C(n,F,\min\limits_{\mathbb{S}^n}r,\max\limits_{\mathbb{S}^n}|\nabla_{\mathbb{S}^n}r|)\int_{\hat{\Sigma}_{t_{\varepsilon}}} \left(F^0(x)F(\nu)-x\cdot\nu\right) \mathrm{d}\mu\\
			&\leq C_0(n,F,\min\limits_{\mathbb{S}^n}r,\max\limits_{\mathbb{S}^n}|\nabla_{\mathbb{S}^n}r|)\sqrt{\epsilon}.
		\end{aligned}
	\end{equation}
	where the Poincar\'e inequality on $\mathbb{S}^n$ is used in the first inequality for $a=\dfrac{1}{|\mathbb{S}^n|}\displaystyle\int_{\mathbb{S}^n} \dfrac{r}{\rho} \mathrm{d}\mathbb{S}^n$. We conclude that
	\begin{equation}
		\max\left\lbrace \max\limits_{\mathbb{S}^n}\dfrac{r}{\rho} - a,a-\min\limits_{\mathbb{S}^n}\dfrac{r}{\rho} \right\rbrace \leqslant C_1(n,F,\min\limits_{\mathbb{S}^n}r,\max\limits_{\mathbb{S}^n}|\nabla_{\mathbb{S}^n}r|) \varepsilon^{\frac{1}{2(n+2)}}=:L.
	\end{equation}
 Indeed, on the contrary, if we suppose
	\begin{equation}
		\max\limits_{\mathbb{S}^n}\dfrac{r}{\rho}-a = \dfrac{r}{\rho}(\theta_0)-a > L.
	\end{equation}
	Then for $l:=\dfrac{L}{2\max\limits_{\mathbb{S}^n}\left| \nabla_{\mathbb{S}^n} \frac{r}{\rho} \right|}>0$, there holds 
	\begin{equation}
		\left|\dfrac{r}{\rho}(\xi)-a\right|>\dfrac{L}{2},\ \forall\ \xi\in\ \overline{B}_l(\theta_0)\subset\mathbb{S}^n.
	\end{equation}
	Thus
	\begin{equation}
		C_0\sqrt{\varepsilon}\geqslant \int_{\mathbb{S}^n} \left| \dfrac{r}{\rho}-a \right|^2  \mathrm{d}A \geqslant  \int_{\overline{B}_{l}(\theta)}  \left| \dfrac{r}{\rho}-a \right|^2 \mathrm{d}A \geq C_2L^{n+2}=C_2C_1^{n+2}\sqrt{\epsilon},
	\end{equation}
	which makes a contradiction when $C_1=\left(\frac{2C_0}{C_2}\right)^{\frac{1}{n+2}}$, where $C_2$ depends on $\max\limits_{\mathbb{S}^n}|\nabla_{\mathbb{S}^n}r|$. Hence
	\begin{equation}
		\mathrm{dist}(\hat{\Sigma}_{t_{\varepsilon}},a\mathcal{W}) \leqslant C \|r-a\rho\|_{L^{\infty}(\mathbb{S}^n)} \leqslant  C(n,F,\min\limits_{\mathbb{S}^n}r,\max\limits_{\mathbb{S}^n}|\nabla_{\mathbb{S}^n}r|)\varepsilon^{\frac{1}{2(n+2)}}.  \label{dist2}
	\end{equation}
	After combining (\ref{dist1}) and (\ref{dist2}), we obtain
	\begin{equation}
		\begin{aligned}
			\mathrm{dist}(\Sigma,&a\mathcal{W}) \leqslant \mathrm{dist}(\Sigma,\hat{\Sigma}_{t_{\varepsilon}})+\mathrm{dist}(\hat{\Sigma}_{t_{\varepsilon}},a\mathcal{W})\\
			\leqslant& C(n,F,\rho_{-}(\Omega),\max\limits_{\mathbb{S}^n}\left|\nabla_{\mathbb{S}^n} r \right|,\min\limits_{\Sigma} H_F) \left(\varepsilon^{\frac{1}{2(n+2)}}+\sqrt{\epsilon}\right).
		\end{aligned}
	\end{equation}
\end{proof}

\begin{proof}[Proof of \eqref{fostable} in Theorem \ref{thm-aniso-stability}]
The argumet is exactly the same as the proof of \eqref{L1-fostable}, since the gradient bound condition also leads to $C^0$ and anisotropic area bound of $\Sigma$. By Young's inequality and the fact that $f_2(ab)\geq \min\{\sqrt{a},a^{ \frac{1}{2(n+2)}} \}f_2(b),\ a,b\geq0$, we get the conclusion. 
\end{proof}
In the following, we use the inequailties proved in Theorem \ref{thm-L1-stability} and \ref{thm-aniso-stability} to show the stability of the Weinstock inequality for star-shaped, strictly mean convex domains.

\begin{proof}[Proof of Corollary \ref{cor-weinstock-stability}]
    Let $P$ be the barycenter of $\Omega$. Let $x_i\ (i=1,2,\cdots,n+1)$ be coordinate functions of $\Omega$, then
    \begin{equation}
    	\int_{\Omega} \left(x_i-P_i\right)\mathrm{d}x =0,\ \ |Dx_i|=1,\ \ i=1,2,\cdots,n+1.
    \end{equation}
    We also have
    \begin{equation}
    	\sigma_1(\Omega) \leqslant \dfrac{(n+1)\mathrm{Vol}(\Omega)}{\displaystyle\int_{\partial\Omega}|x-P|^2 \mathrm{d}\mu}
    \end{equation}
    and
    \begin{equation}
    	\sigma_1(\mathbb{B}^{n+1})=1
    \end{equation}
    by definition (\ref{steklov}). After using the isoperimetric inequality 
    \begin{equation}
    	\dfrac{\mathrm{Vol}(\Omega)^{\frac{n-1}{n+1}}}{|\partial\Omega|^{\frac{n-1}{n}}} \leqslant \dfrac{\mathrm{Vol}(\mathbb{B}^{n+1})^{\frac{n-1}{n+1}}}{|\mathbb{S}^n|^{\frac{n-1}{n}}},
    \end{equation}
    we obtain
    \begin{eqnarray}
    	&&\sigma_1(\mathbb{B}^{n+1}) \left|\mathbb{S}^n\right|^{\frac{1}{n}} - \sigma_1(\Omega) \left|\partial \Omega\right|^{\frac{1}{n}} \notag\\
    	&\geqslant& \left|\mathbb{S}^n\right|^{\frac{1}{n}}-\dfrac{(n+1)\mathrm{Vol}(\Omega)|\partial\Omega|^{\frac{1}{n}}}{\displaystyle\int_{\partial\Omega}|x-P|^2\mathrm{d}\mu}\notag\\
    	&\geqslant&  (n+1)\dfrac{\mathrm{Vol}(\mathbb{B}^{n+1})^{\frac{n-1}{n+1}}}{\left|\mathbb{S}^n\right|^{\frac{n-1}{n}}} \left[ \mathrm{Vol}(\mathbb{B}^{n+1})^{\frac{2}{n+1}} - \dfrac{|\partial\Omega|\mathrm{Vol}(\Omega)^{\frac{2}{n+1}}}{\displaystyle\int_{\partial\Omega} |x-P|^2 \mathrm{d}\mu } \right]. \label{s1}
    \end{eqnarray}
   Substituting $F^0(x-P)=|x-P|$, $p=2$ in \eqref{L1-fostable} and $P=o$ in (\ref{fostable}), combining \eqref{ineq-aniso2'} in isotropic case:
    $$\int_{\partial\Omega}|x-P|^2\mathrm{d}\mu\geqslant \mathrm{Vol}(\mathbb{B}^{n+1})^{-\frac{2}{n+1}}|\partial\Omega|\mathrm{Vol}(\Omega)^{\frac{2}{n+1}},$$
    we conclude that
      \begin{equation}\label{s2}
  	\mathrm{Vol}(\mathbb{B}^{n+1})^{\frac{2}{n+1}} - \dfrac{|\partial\Omega|\mathrm{Vol}(\Omega)^{\frac{2}{n+1}}}{\displaystyle\int_{\partial\Omega} |x-P|^2 \mathrm{d}\mu } \geqslant \left\{\begin{aligned}
  		&C\cdot f_1^{-1}\left[\alpha(\Omega)\right],\\
  		&C\cdot f_2^{-1}\left[\mathrm{dist}(\partial\Omega,a\mathbb{S}^n)\right].
  	\end{aligned}\right. 
  \end{equation}
   Combining (\ref{s1}) and (\ref{s2}), we obtain the desired conclusions \eqref{L1-weinstable} and (\ref{weinstable}).
\end{proof}

\begin{ack}
The authors would like to thank Professor Yong Wei for his helpful discussions and constant support. The authors were supported by National Key Research and Development Program of China 2021YFA1001800.
\end{ack}

\end{document}